\documentclass[10pt, a4paper]{amsart}
\usepackage[utf8]{inputenc}
\usepackage{amsfonts}
\usepackage{amssymb}
\usepackage{amsmath}
\usepackage{newlfont}
\usepackage[mathcal]{euscript}
\usepackage{tikz}
\usepackage{enumerate}
\usepackage{hyperref}
\usepackage{amsthm}
\usetikzlibrary{decorations.pathreplacing}
\usetikzlibrary{patterns}

\usetikzlibrary{arrows,shapes,positioning}
\usetikzlibrary{decorations.markings}
\tikzstyle arrowstyle=[scale=1]
\tikzstyle directed=[postaction={decorate,decoration={markings,
    mark=at position .65 with {\arrow[arrowstyle]{stealth}}}}]
\tikzstyle reverse directed=[postaction={decorate,decoration={markings,
    mark=at position .65 with {\arrowreversed[arrowstyle]{stealth};}}}]

 \newtheorem{question}{Question}

 \newtheorem{thm}{Theorem}
\newtheorem{defn}[thm]{Definition}
\newtheorem{lem}[thm]{Lemma}

\newtheorem{obs}[thm]{Observation}

\newtheorem*{remark*}{Remark}
\newtheorem{definition}[thm]{Definition}
\newtheorem{lemma}[thm]{Lemma}
\newtheorem{proposition}[thm]{Proposition}

\newtheorem{remark}[thm]{Remark}

\DeclareMathOperator{\diam}{diam}

\newcommand{\supp}{\mathrm{supp}}
\newcommand{\per}{\mathrm{Per}}
\newcommand{\fix}{\mathrm{Fix}}

\newcommand{\crit}{\mathrm{Crit}}
\newcommand{\co}{\mathrm{CO}}

 \newcommand{\eps}{\varepsilon}

 \newcommand{\N}{\mathbb{N}}

 \newcommand{\set}[1]{\left\{#1\right\}}

 \def\P{{\mathcal P}}

\numberwithin{equation}{section}

\begin{document}


\title[Periodic points and shadowing property for $C_\lambda(I)$]{Periodic points and shadowing property for generic Lebesgue measure preserving interval maps}

\author{Jozef Bobok}
\author{Jernej \v Cin\v c}
\author{Piotr Oprocha}
\author{Serge Troubetzkoy}

\address[J.\ Bobok]{Department of Mathematics of FCE, Czech Technical University in Prague, 
Th\'akurova 7, 166 29 Prague 6, Czech Republic}
\email{jozef.bobok@cvut.cz}

\date{\today}

\subjclass[2020]{37E05, 37A05, 37B65, 37C20}
\keywords{Interval map, circle map, Lebesgue measure preserving, periodic points, Hausdorff dimension, upper box dimension, shadowing, generic properties}

\address[J.\ \v{C}in\v{c}]{Faculty of Mathematics, University of Vienna, Oskar-Morgenstern-Platz 1,
A-1090 Vienna, Austria. -- and -- Centre of Excellence IT4Innovations - Institute for Research and Applications of Fuzzy Modeling, University of Ostrava, 30. dubna 22, 701 03 Ostrava 1, Czech Republic}
\email{jernej.cinc@osu.cz}

\address[P.\ Oprocha]{AGH University of Science and Technology, Faculty of Applied Mathematics, 
al.\ Mickiewicza 30, 30-059 Krak\'ow, Poland. -- and -- 
Centre of Excellence IT4Innovations - Institute for Research and Applications of Fuzzy Modeling, University of Ostrava, 30. dubna 22, 701 03 Ostrava 1, Czech Republic}
\email{oprocha@agh.edu.pl}

\address[S.\ Troubetzkoy]{Aix Marseille Univ, CNRS, Centrale Marseille, I2M, Marseille, France
postal address: I2M, Luminy, Case 907, F-13288 Marseille Cedex 9, France}
\email{serge.troubetzkoy@univ-amu.fr}

\begin{abstract}
    We show that for the generic continuous maps of the interval and circle which preserve the Lebesgue measure it holds for each $k\ge 1$ that the set of periodic points of period $k$ is a Cantor set of Hausdorff dimension zero and of upper box dimension one. Furthermore, building on this result, we show that there is a dense collection of transitive Lebesgue measure preserving interval maps whose periodic points have full Lebesgue measure and whose periodic points of period $k$ have positive measure for each $k \ge 1$. 
	Finally, we show that the generic continuous maps of the interval which preserve the Lebesgue measure  
	satisfy the shadowing and periodic shadowing property.
\end{abstract}
\maketitle
\section{Introduction}

In what follows let a {\em residual} set be a dense $G_{\delta}$ set and we call a property {\em generic} if it is satisfied on at least a residual set of the underlying space.
The roots of studying generic properties in dynamical systems can be derived from the article by Oxtoby and Ulam from 1941 \cite{OxUl} in which they showed that for a finite-dimensional compact manifold with a non-atomic measure which is positive on open sets, the set of ergodic measure-preserving homeomorphisms is generic in the strong topology.
Subsequently, Halmos in 1944  \cite{Ha44},\cite{Ha44.1} introduced approximation techniques to a purely metric situation: the study of interval maps which are invertible almost everywhere and preserve the Lebesgue measure and showed that the generic invertible map is weakly mixing, i.e., has continuous spectrum.
Then, Rohlin in 1948 \cite{Ro48} showed that the set of (strongly) mixing measure preserving invertible maps is of the first category.
Two decades later, Katok and Stepin in 1967  \cite{KaSt}  introduced the notation of a speed of approximations. One of the notable applications of their method is the genericity of ergodicity and weak mixing for certain classes of interval exchange transformations.
One of the most outstanding result using approximation theory is the Kerckhoff, Masur, Smillie result on the existence of polygons for which the billiard flow is ergodic \cite{KeMaSm}, as well as its quantitative version by Vorobets \cite{Vo}.
Many more details on the history of approximation theory can be found in the surveys
\cite{BeKwMe}, \cite{ChPr},  \cite{Tr}.

In what follows we denote $I:=[0,1]$,  $\mathbb{S}^1$ the unit circle and $\lambda$ the Lebesgue measure on an underlying manifold. Our present study focuses on topological properties of generic non-invertible maps on the interval resp.\  circle preserving the Lebesgue measure $C_{\lambda}(I)$, resp.\  $C_{\lambda}(\mathbb{S}^1)$. For the rest of the paper we equip the two spaces with the uniform metric, which makes the spaces complete. The study of generic properties on $C_{\lambda}(I)$ was initiated in \cite{B} and continued recently in \cite{BT}. It is well known that every such map has a dense set of periodic points (see for example \cite{BT}).  Furthermore, except for the two exceptional maps $\mathrm{id}$ and $1-\mathrm{id}$, every such map has positive metric entropy. Recently, basic topological and measure-theoretical properties of generic maps from $C_\lambda(I)$ were studied in \cite{BT}. We say that an interval map $f$ is \emph{locally eventually onto (leo)} if for every open interval $J\subset I$ there exists a non-negative integer $n$ so that $f^n(J)=I$. This property is also sometimes referred in the literature as \emph{topological exactness}.
The $C_{\lambda}(I)$-generic function
\begin{enumerate}[(a)]
	\item is weakly mixing with respect to $\lambda$ \cite[Th. 15]{BT},
	\item is leo \cite[Th. 9]{BT},
	\item satisfies the periodic specification property \cite[Cor. 10]{BT},
	\item has a knot point at $\lambda$ almost every point \cite{B},
	\item maps a set of Lebesgue measure zero onto $[0,1]$ \cite[Cor. 22]{BT},
	\item has infinite topological entropy \cite[Prop. 26]{BT},
	\item has Hausdorff dimension = lower Box dimension = 1 $<$ upper Box dimension  = 2 \cite{SW95}.
\end{enumerate}
 It was furthermore shown that the set of mixing maps in $C_{\lambda}(I)$ is dense  \cite[Cor. 14]{BT} and in analogy to Rohlin's result \cite{Ro48} that this set is of the first category \cite[Th. 20]{BT}. 
 
 In this paper we delve deeper in the study of properties of generic Lebesgue measure preserving maps on manifolds of dimension $1$.
 Our choice of $C_{\lambda}(I)$ for further investigation is that they are one-dimensional versions of volume-preserving maps, or more broadly, conservative dynamical systems. On the other hand, they represent variety of possible one-dimensional dynamics as highlighted in the following.

\begin{remark*}Let $f$ be an interval map. The following conditions are equivalent.
\begin{itemize}
 \item[(i)] $f$ has a dense set of periodic points, i.e., $\overline{\per(f)}=I$.
 \item[(ii)] $f$ preserves a nonatomic probability measure $\mu$ with $\supp~\mu=I$. 
    \item[(iii)] There exists a homeomorphism $h$ of $I$ such that $h\circ f\circ h^{-1}\in C_{\lambda}(I)$.
\end{itemize}
\end{remark*}
To see the above equivalence it is enough to combine a few facts from the literature. The starting point is \cite{BM}, where the dynamics of interval maps with dense set of periodic points had been described; while this article is purely topological it easily implies that such maps must have non-atomic invariant measures with full support.
The Poincar\'e Recurrence Theorem and the fact that in dynamical system given by an interval map the closures of recurrent points and periodic points coincide \cite{CoHe80} provides connection between maps preserving 
a probability measure with full support and dense set of periodic points. Finally, for $\mu$ a non-atomic probability measure with full support the map $h\colon~I\to I$ defined as $h(x)=\mu([0,x])$ is a homeomorphism of $I$; moreover, if $f$ preserves $\mu$ then $h\circ f\circ h^{-1}\in C_{\lambda}(I)$ (see the proof of Theorem~\ref{t-PP} for more detail on this construction). Therefore, the topological properties that are proven in \cite{BT} and later in this paper are generic also for interval maps preserving measure $\mu$.

A basic tool to understand the dynamics of interval maps is to understand the structure, dimension and Lebesgue measure of the set of its periodic points.
For what follows let $f\in C_{\lambda}(I)$. Since generic maps from $C_{\lambda}(I)$ are weakly mixing with respect to $\lambda$ it holds that the Lebesgue measure on the periodic points is $0$.  However, it is still natural to ask: 

\begin{question}
	\textit{What is the cardinality, structure and dimension of periodic points for generic maps in $C_{\lambda}(I)$?}
\end{question}

Akin et.~al.~proved in \cite[Theorems 9.1 and 9.2(a)]{AHK} that the set of periodic points of generic homeomorphisms of $\mathbb{S}^1$ is a Cantor set.
In an unpublished sketch, Guih\'eneuf showed that the set of periodic points of a generic
volume preserving homeomorphism $f$
of a manifold of dimension at least two (or more generally preserving a good
measure in the sense of Oxtoby and Ulam \cite{OxUl})
 is a dense set of measure zero and for any $\ell \ge 1$ the set of fixed points of $f^{\ell}$  is either empty or a perfect set  \cite{PAG}.
On the other hand, Carvalho et.~al.~have shown that  the upper box dimension of the set of periodic points  is full for generic homeomorphisms on compact manifolds of dimension at least one  \cite{PV}\footnote{this statement only appears in the published version of \cite{PV}.}.

In the above context we provide the general answer about the cardinality and structure of periodic points of period $k$ for $f$ (denoted by $\per(f,k)$), of fixed points of $f^k$ (denoted by $\fix(f,k)$) and of the union of all periodic points of $f$ (denoted by $\per(f)$) and its respective lower box, upper box and Hausdorff dimensions. Namely, we prove:

\begin{thm}\label{t8}
	For a generic map $f \in C_{\lambda}(I)$,  for each $k \geq 1$:
	\begin{enumerate}
		\item \label{pp1} the set $\fix(f,k)$ is  a Cantor set, 
		\item  \label{pp2} the set $\per(f,k)$  is  a Cantor set,  
		\item \label{pp3} the set $\fix(f,k)$ has Hausdorff dimension and  lower box dimension zero. In particular, $\per(f,k)$ has Hausdorff dimension and lower box dimension zero.
		\item \label{pp4}  the set $\per(f,k)$ has upper box dimension one. Therefore, $\fix(f,k)$ has upper box dimension one as well.
		\item \label{pp5} the Hausdorff dimension of $\per(f)$ is zero.
	\end{enumerate}
\end{thm}

The proof of the above theorem works also for the generic continuous maps which by our knowledge is not known in the literature yet.
Furthermore, we can also address the setting of $C_{\lambda}(\mathbb{S}^1)$, however, due to the presence of rotations, we need to treat degree $1$ maps separately (for the related statement of the degree one case we refer the reader to Theorem~\ref{thm:C_p}).

Related to the study above, there is an interesting question about the possible Lebesgue measure on the set of periodic points for maps from $C_{\lambda}(I)$. 

\begin{question}\label{q:B}
	\textit{Does there exist a transitive (or even leo) map in $C_{\lambda}(I)$ with positive Lebesgue measure on the set of periodic points?}
\end{question}

As mentioned already above, generic maps from $C_{\lambda}(I)$ will have Lebesgue measure $0$ since  $\lambda$ is weakly mixing. Therefore, the previous question asks about the complement of generic maps from $C_{\lambda}(I)$ and requires on the first glance contradicting properties. The discrepancy between topological and measure theoretical aspect of dynamical systems again comes to display and we obtain the following result. We answer Question~\ref{q:B} and even prove a stronger statement.

\begin{thm}\label{t-PP}
	The set of leo maps in $C_{\lambda}(I)$
	whose periodic points have full Lebesgue measure and whose periodic points of period $k$ have positive measure for each $k \ge 1$
	is dense in $C_{\lambda}(I)$.
\end{thm}

Another motivation for the study in this paper was the following natural question. 

\begin{question}\label{q3}
	\textit{Is shadowing property generic in $C_{\lambda}(I)$?}
\end{question}

Shadowing is a classical notion in topological dynamics and it serves as a tool to determine whether any hypothetical orbit is actually close to some real orbit of a topological dynamical system; this is of great importance in systems with sensitive dependence on initial conditions, where small errors may potentially result in a large divergence of orbits.
Pilyugin and Plamenevskaya introduced in \cite{Pi} a nice technique to prove that shadowing is generic for homeomorphisms on any smooth compact manifold without a boundary. This led to several subsequent results that shadowing is generic in topology of uniform convergence, also in dimension one (see \cite{Med, KoMaOpKu} for recent results of this type). On the other hand, there are many cases known, when shadowing is not present in an open set in $C^1$ topology (see survey paper by Pilyugin \cite{PilSur} and \cite{PilBook,PilBook2} for the general overview on the recent progress related with shadowing).

For continuous maps on manifolds of dimension one, Mizera proved that shadowing is indeed a generic property \cite{Mizera}. In the context of volume preserving homeomorphisms on manifolds of dimension at least two (with or without boundary), the question above was solved recently in the affirmative by Guih\'eneuf and  Lefeuvre \cite{GuLe18}. 

Our last main theorem provides the affirmative answer on Question~\ref{q3}.

\begin{thm}\label{t-pshadow}
	Shadowing and periodic shadowing are generic properties for maps from $C_{\lambda}(I)$.
\end{thm}

Let us briefly describe the structure of the paper. In Preliminaries we give general definitions that we will need in the rest of the paper. In particular, our main tool throughout the most of the paper will be controlled use of approximation techniques which we introduce in the end of Section~\ref{sec:Preliminaries}.  In Section~\ref{sec:PP} we turn our attention to the study of periodic points and prove Theorem~\ref{t8}.
The proof relies on a precise control of perturbations introduced in Section~\ref{sec:Preliminaries} which turns out to be particularly delicate.
 With some additional work we consequently obtain Theorem~\ref{t-PP}. We conclude the section with the study of periodic points for maps from $C_{\lambda}(\mathbb{S}^1)$ in Subsection~\ref{subsec:PPcirc}. 
In Section~\ref{sec:shadowing} we provide a proof of Theorem~\ref{t-pshadow}. Similarly as in \cite{KoMaOpKu} we use covering relations, however the main obstacle is the preservation of Lebesgue measure which makes obtaining such coverings a more challenging task. We conclude the paper with Subsection~\ref{subsec:s-limit} where we address a notion stronger than shadowing called the s-limit shadowing (see Definition~\ref{def:shadowing}) in the contexts of $C_{\lambda}(I)$. We prove that s-limit shadowing is dense in the respective environments. The approach resembles the one taken in \cite{MazOpr}, however due to our more restrictive setting our proof requires better control of perturbations. This result, in particular, implies that limit shadowing is dense in the respective environments as well.

\section{Preliminaries}\label{sec:Preliminaries}
Let $\mathbb{N}:=\{1,2,3,\ldots\}$ and $\mathbb{N}_0:=\mathbb{N}\cup\{0\}$.
Let $\lambda$ denote the Lebesgue measure on the unit interval $I:=[0,1]$. We denote by $C_{\lambda}(I)\subset C(I)$ the family of all continuous Lebesgue measure preserving functions of $I$ being a proper subset of the family of all continuous interval maps equipped with the \emph{uniform metric} $\rho$:
$$\rho (f,g) := \sup_{x \in I} |f(x) - g(x)|.$$
A \emph{critical point} of $f$ is a point $x\in I$ such that there exists no neighborhood of $x$ on which $f$ is strictly monotone. Denote by $\crit(f)$ the set of all critical points of $f$. A point $x$ is called \emph{periodic of period $N$}, if there exists $N\in\N$ so that $f^N(x)=x$ and we take the least such $N$. Let us denote by $\Xi(f)$ the set of points from $I$ for which no neighborhood has a constant slope under $f$. Obviously, $\mathrm{Crit}(f)\subset \Xi(f)$. Let $\mathrm{PA}(I)\subset C(I)$ denote the set of \emph{piecewise affine} 
functions; i.e., functions that are affine on every interval of monotonicity and have finitely many points in the set $\Xi(f)$.
 Let $\mathrm{PA}_{\lambda}(I)\subset C_{\lambda}(I)$ denote the set of piecewise affine functions that preserve Lebesgue measure and $\mathrm{PA}_{\lambda(\textrm{leo})}(I)\subset\mathrm{PA}_{\lambda}(I)$ such functions that are additionally \emph{locally eventually onto} (i.e., the image under sufficiently large  iterations of  nonempty open sub-intervals  cover $I$). For some $\xi>0$ and $h\in C_{\lambda}(I)$ define 
 $$B(h,\xi):=\{f\in C_{\lambda}(I): \rho(f,h)<\xi\}.$$
 By $d(x,y)$ we denote the Euclidean distance on $I$ between $x,y\in I$.
 For a set $U\in I$ and $\xi>0$ we will also often use, by the abuse of notation, 
 $$B(U,\xi):=\{x\in I: d(u,x)<\xi \text{ for some } u\in U\}.$$

\subsection{Window perturbations in Lebesgue preserving setting} In this subsection we briefly discuss the setting of Lebesgue measure preserving interval maps and introduce the techniques that we will apply in the rest of the paper.
The proof  of the following proposition is standard and we leave it for the reader.

\begin{proposition}\label{prop:1}$(C_{\lambda}(I),\rho)$  is a complete metric space. \end{proposition}

\begin{definition} We say that continuous maps $f,g:[a,b]\subset I\to I$ are \emph{$\lambda$-equivalent} if for each Borel set $A \in \mathcal{B}$,
$$\lambda(f^{-1}(A))=\lambda(g^{-1}(A)).
$$
For $f\in C_{\lambda}(I)$ and $[a,b]\subset I$ we denote by $C(f;[a,b])$ the set of all continuous maps $\lambda$-equivalent to $f| [a,b]$. We define
$$C_*(f;[a,b]):=\{h\in C(f;[a,b]):h(a)=f(a),~h(b)=f(b)\}.$$
\end{definition}

The following definition is illustrated by Figure~\ref{fig:perturbations}.
 
\begin{definition}\label{eq:1} Let $f$ be from $C_{\lambda}(I)$ and $[a,b]\subset I$. For any fixed $m\in\N$, let us define the map $h=h\langle f;[a,b],m\rangle\colon~[a,b]\to I$ by ($j\in\{0,\dots,m-1\}$):
\begin{equation*}
h(a + x) := \begin{cases}
f\left (a+m \Big (x-\frac{j(b-a)}{m}\Big) \right )\text{ if } x\in \left [\frac{j(b-a)}{m},\frac{(j+1)(b-a)}{m} \right ],~j\text{ even}, \\
f\left (a+m \Big (\frac{(j+1)(b-a)}{m}-x \Big ) \right )\text{ if } x\in \left [\frac{j(b-a)}{m},\frac{(j+1)(b-a)}{m} \right ],~j\text{ odd}.
\end{cases}
\end{equation*}
Then $h\langle f;[a,b],m\rangle\in C(f;[a,b])$ for each $m$ and $h\langle f;[a,b],m\rangle\in C_*(f;[a,b])$ for each $m$ odd.
\end{definition}
 
\begin{figure}[!ht]
	\centering
	\begin{tikzpicture}[scale=4]
	\draw (0,0)--(0,1)--(1,1)--(1,0)--(0,0);
	\draw[thick] (0,1)--(1/2,0)--(1,1);
	\node at (1/2,1/2) {$f$};
	\node at (7/16,-0.1) {$a$};
	\node at (5/8,-0.1) {$b$};
	\draw[dashed] (7/16,0)--(7/16,1/4)--(5/8,1/4)--(5/8,0);
	\node[circle,fill, inner sep=1] at (7/16,1/8){};
	\node[circle,fill, inner sep=1] at (5/8,1/4){};
	\end{tikzpicture}
	\hspace{1cm}
	\begin{tikzpicture}[scale=4]
	\draw (0,0)--(0,1)--(1,1)--(1,0)--(0,0);
	\draw[thick] (0,1)--(7/16,1/8)--(14/32+1/48,0)--(1/2,1/4)--(1/2+1/24,0)--(9/16,1/8)--(9/16+1/48,0)--(5/8,1/4)--(1,1);
	\draw[dashed] (1/2,0)--(1/2,1/4);
	\draw[dashed] (9/16,1/4)--(9/16,0);
	\node at (1/2,1/2) {$h$};
	\node at (7/16,-0.1) {$a$};
	\node at (5/8,-0.1) {$b$};
	\draw[dashed] (7/16,0)--(7/16,1/4)--(5/8,1/4)--(5/8,0);
	\node[circle,fill, inner sep=1] at (1/2,1/4){};
	\node[circle,fill, inner sep=1] at (9/16,1/8){};
	\node[circle,fill, inner sep=1] at (7/16,1/8){};
	\node[circle,fill, inner sep=1] at (5/8,1/4){};
	\end{tikzpicture}
	\caption{For  $f\in C_{\lambda}(I)$ shown on the left, on the right
we show the the regular $3$-fold window perturbation of $f$ by $h=h\langle f;[a,b],3\rangle\in C_*(f;[a,b])$.}\label{fig:perturbations}
\end{figure}
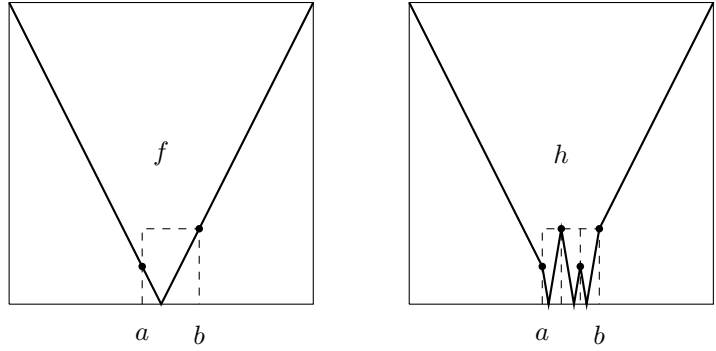

For more details on the perturbations from the previous definition we refer the reader to \cite{BT}.

\begin{definition}\label{def:perturb}
For a fixed $h\in C_*(f;[a,b])$, the map $g=g\langle f,h\rangle\in C_{\lambda}(I)$ defined by
\begin{equation*}
g(x) := \begin{cases}
f(x)\text{ if } x\notin [a,b],\\
h(x)\text{ if } x\in [a,b]
\end{cases}
\end{equation*}
will be called the \emph{window perturbation} of $f$ (by $h$ on $[a,b]$). In particular, if $h=h\langle f;[a,b],m\rangle$, $m$ odd, (resp. $h$ is piecewise
affine), we will speak of
\emph{regular $m$-fold (resp. piecewise affine) window perturbation} $g$ of $f$ (on $[a,b]$).
\end{definition}

\section{Cardinality and dimension of periodic points for generic Lebesgue measure preserving interval and circle maps}\label{sec:PP}
Since generic maps from $C_{\lambda}(I)$ are weakly mixing (Theorem 15 from \cite{BT}) it follows that the Lebesgue measure of the periodic points of generic maps from $C_{\lambda}(I)$ is $0$. The main result of this section is Theorem \ref{t8} which describes the structure, cardinality and dimensions of this set.

Let 
$$\fix(f,k) := \{x: f^k(x) = x\}$$
$$\per(f,k) := \{x : f^k(x) = x \text{ and } f^i(x) \ne x  \text{ for all } 1 \le i < k
\}$$
$$k(x) := k \text{ for } x \in \per(f,k)$$
and
$$\per(f) := \bigcup_{k \ge 1} \per(f,k) = \bigcup_{k \ge 1} \fix(f,k).$$

\begin{definition}\label{transverse}
A periodic point $p \in \per(f,k) $ is called \emph{transverse} if  there exist three adjacent intervals
$A = (a_1,a_2) ,B = [a_2,c_1] ,C = (c_1,c_2)$, with $p \in B$, $B$ possibly reduced to a point,
such that (1)  $f^k(x) = x$ for all $x \in B$ and either (2.a) $f^k(x) > x$ for all $x \in A$ and $f^k(x) < x$ for all $x \in C$ or (2.b)  $f^k(x) <x$ for all $x \in A$ and $f^k(x) > x$ for all $x \in C$.
\end{definition}

To prove Theorem~\ref{t8} we will use the following lemma.

\begin{lemma} For each $k \ge 1$ 
there is a dense set $\{g_i\}_{i\geq 1}$  of maps in $C_{\lambda}(I)$ such that
$g_i \in \mathrm{PA}_{\lambda}(I)$, $\per(g_i,k) \not = \emptyset$, and for each $i$ all points in $\fix(g_i,k)$ are transverse.
\end{lemma}

\begin{proof}
The set $\mathrm{PA}_{\lambda}(I)$ is dense in $C_{\lambda}(I)$ (\cite{B}, see also  Proposition 8 in \cite{BT}). 
Each $f \in \mathrm{PA}_{\lambda}(I)$ (in fact each $f \in C_{\lambda}(I)$) has a fixed point,  so using a 3-fold window perturbation around the fixed point we can approximate $f$ arbitrarily
well by a map $f_1 \in\mathrm{PA}_{\lambda}(I)$ with $\per(f_1,k) \not = \emptyset$.

Fix $f \in\mathrm{PA}_{\lambda}(I)$ with $\per(f,k) \not = \emptyset$.
We claim that by an arbitrarily small perturbation of $f$
 we can construct  a map $g \in\mathrm{PA}_{\lambda}(I)$ such that
 \begin{equation}
 \per(g,k) \not = \emptyset \text{ and
 all points in } \fix(g,k) \text{  are transverse.}\label{e1}
 \end{equation}
 
 We do this in several steps. The first step is to perturb $f$ to $g$ in such a way that the points $0$ and $1$ are not in $\fix(g,k)$;
 we will treat only the point $0$, the arguments for the point $1$ are analogous.
If $0$ is a fixed point we can make an arbitrarily small window perturbation as in Figure
\ref{fig:Per1} so that this is no longer the case.

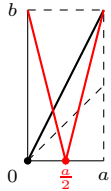
\begin{figure}[!ht]
	\begin{tikzpicture}[scale=2]
	\draw (0,1)--(0,0)--(1/2,0);
	\draw[dashed] (0,1)--(1/2,1)--(1/2,0);
	\draw[dashed] (0,0)--(1/2,1/2);
	\draw[thick] (0,0)--(1/2,1);
	\draw[thick,red] (0,1)--(0.25,0)--(1/2,1);
	\node at (-0.1,-0.1) {$\scriptstyle 0$};
	\node at (1/2,-0.1) {$\scriptstyle a$};
	\node at (0.25,-0.12) {{\color{red}$\scriptstyle \frac{a}{2}$}};
	\node[circle,fill,red, inner sep=1] at (0.25,0){};
	\node[circle,fill, inner sep=1] at (0,0){};
	\node at (-0.1,1) {$\scriptstyle b$};
	\end{tikzpicture}
	\caption{Small perturbations of a map $f\in \mathrm{PA}_{\lambda}(I)$ near 0. }\label{fig:Per1}
\end{figure}

Now consider the case when  $f^{j}(0) =0$, where $j > 1$ is
the period of the point $0$, and $j |k$. We assume that $a$ is so small that $f^i([0,a]) \cap [0,a] = \emptyset$ for $i=1, 2, \dots,  j-1$
and choose $a$ so that $f^j(a) \ne 0$. Let $g$ be the map resulting from a regular 2-fold window perturbation of $f$ on the interval $[0,a]$
(see Figure \ref{fig:Per2}). Thus $g(0) = f(a)$ and  $g^{j -1} = f^{j -1} $ on the interval $[f(0), f(a)]$, and so 
$g^j(0) = g^{j-1} \circ f(a) =  f^j(a) \ne 0$.

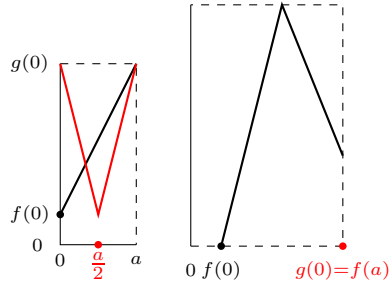
\begin{figure}[!ht]
	\begin{tikzpicture}[scale=2]
	\draw (0,1)--(0,-0.2)--(1/2,-0.2);
	\draw[dashed] (0,1)--(1/2,1)--(1/2,-0.2);
	\draw[thick] (0,0)--(1/2,1);
	\draw[thick,red] (0,1)--(0.25,0)--(1/2,1);
	\node at (-0.2,0) {$\scriptstyle f(0)$};
	\node at (-0.2,1) {$\scriptstyle g(0)$};
	\node at (0,-0.3) {$\scriptstyle 0$};
	\node at (-0.15,-0.2) {$\scriptstyle 0$};
	\node at (1/2,-0.3) {$\scriptstyle a$};
	\node at (0.25,-0.34) {{\color{red}$\frac{a}{2}$}};
	\node[circle,fill,red, inner sep=1] at (0.25,-0.2){};
	\node[circle,fill, inner sep=1] at (0,0){};
	\end{tikzpicture}
	\hspace{0.1cm}
	\begin{tikzpicture}[scale=4]
	\draw[dashed] (.1, 1/5) --(3/5,1/5) -- (3/5,1) -- (.1,1);
	\draw (.1, 1/5)  -- (.1,1) ;
	\draw [thick] (1/5,1/5)--(2/5,1) -- (3/5,1/2);

	\node at (.2, 0.12) {$\scriptstyle f(0)$};
	\node[circle,fill, inner sep=1] at (.2,0.2){};
	\node[red] at (.6, 0.12) {$\scriptstyle g(0)=f(a)$};
	\node[circle,fill, red, inner sep=1] at (0.6,0.2){};
	\node at (0.1,0.12) {$\scriptstyle 0$};
	\end{tikzpicture}

	\caption{Left: small perturbations of a map $f\in \mathrm{PA}_{\lambda}(I)$ near 0; Right:  $g^{j -1} =f^{j-1}$ on $[f(0),f(a)]$.}\label{fig:Per2}
\end{figure}

Thus we can choose a dense set of 
 $f \in\mathrm{PA}_{\lambda}(I)$ with $\per(f,k) \not = \emptyset$,
and $\fix(f,k) \cap \{0,1\} = \emptyset$. Fix such a map $f$.
We claim that by an arbitrarily small perturbation of $f$
 we can construct a $g \in\mathrm{PA}_{\lambda}(I)$ with $\per(g,k) \not = \emptyset$ and $\fix(g,k) \cap \{0,1\} = \emptyset$  such that for every 
  $c \in \crit(g)$ we have $g^i(c) \not \in \crit(g)$ for all $1\leq i \le k$.

Suppose 
that for some $c_1,c_2 \in \crit(f)$ we have an $\ell\geq 1$ such that $f^{\ell}(c_1) = c_2$
and $f^i(c_1) \not \in \crit(f)$ for $1 \le i \le \ell -1$. 
We call this orbit a \textit{critical connection of length $\ell$}. 
Choose $c_1,c_2$ with the minimal such $\ell$, if there are several choices fix one of them.
We will perturb $f$ to a map $g$ for which this critical connection is destroyed, so $g$ has
one less critical connection of length $\ell$.
Since there are finitely many critical connections of a given length, a finite number of such
perturbations will remove all of them, and a countable sequence of perturbations will  finish the proof
of the claim.

If $c_1 \ne c_2$ it suffices to use a small window perturbation around $c_2$ as in Figure
\ref{fig:Per};
if the window perturbation is disjoint from the orbit segment $f^i(c_1)$ for $i \in \{0,1,\dots, \ell -1\}$
then for the resulting  map $g$ we  have $g^{\ell}(c_1) = c_2 < \bar{c}_2$ and $g^i(c_1) \not \in\crit(g) = \{\bar{c}_2\} \cup \crit(f) \setminus \{c_2\}$  for $i=1,2,\dots,\ell - 1$, 
thus we have destroyed the critical connection.
 Let $Q=\{f^i(c): 0 \le i< \ell \text{ and } c\in \crit(f)\} \setminus \{c_2\}$. 
Notice that $g(\bar{c}_2) =f(c_2)$, thus taking the neighborhood for the perturbations sufficiently small to be disjoint from $Q$
guaranties that we did not create a new critical connections of length $\ell$ or shorter.

\begin{figure}[!ht]
	\begin{tikzpicture}[scale=4]
	\draw[dashed] (1/10,1/5)--(1/10,3/5)--(3/10,3/5)--(3/10,1/5)--(1/10,1/5);
	\draw [thick] (1/10,3/5)--(2/10,1/5)--(3/10,3/5);
	\draw [thick,red] (1/10,3/5)--(5/20,1/5)--(3/10,3/5);
	\node[circle,fill, inner sep=1] at (0.2,0.2){};
	\node[circle,fill,red, inner sep=1] at (0.25,0.2){};
	\node at (0.2,0.12) {$c_2$};
	\node at (0.28,0.12) {{\color{red}$\bar{c}_2$}};
	\end{tikzpicture}
	\caption{Small perturbations of a map $f\in \mathrm{PA}_{\lambda}(I)$.}\label{fig:Per}
\end{figure}
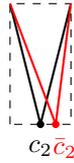

 Now consider the case $c_1 = c_2$.
Suppose  that $c_1$ is a local minimum of $f$; the other cases are
similar. If $\ell = 1$ then we again move the peak using the window perturbation as in Figure \ref{fig:Per} to destroy
the connection.  If $\ell > 1$ then by assumption the map $f^{\ell-1}$ in a neighborhood $U = (a,b)$ of the point $f(c_1)$ is strictly monotone.
Using a window perturbation around $c_1$ as in Figure~\ref{fig:Pert},
yields a map $g \in C_{\lambda}(I)$ with $\per(g,k) \not = \emptyset$
such that
$$\crit(g) = \{\bar{c}_1\} \cup \crit(f) \setminus \{c_1\}.$$  
\begin{figure}[!ht]
	\begin{tikzpicture}[scale=4]
	\draw[dashed] (1/10,1/5)--(1/10,3/5)--(3/10,3/5)--(3/10,1/5)--(1/10,1/5);
	\draw [thick] (1/10,3/5)--(2/10,1/5)--(3/10,3/5);
	\draw [thick,red] (1/10,3/5)--(5/20,1/5)--(3/10,3/5);
	\node[circle,fill, inner sep=1] at (0.2,0.2){};
	\node[circle,fill,red, inner sep=1] at (0.25,0.2){};
	\node at (0.2,0.12) {$\scriptstyle c_1$};
	\node at (0.28,0.12) {{\color{red}$\scriptstyle \bar{c}_1$}};
	\node at (0.28,-0.05) {\phantom{1}};
	\end{tikzpicture}
	\hspace{0.1cm}
	\begin{tikzpicture}[scale=4]
	\draw[dashed] (1/5, 1/5) -- (3/5,1/5) -- (3/5,1) -- (1/5,1) -- (1/5,1/5);
	\draw [thick] (1/5,1/5)--(3/5,1);
	\node at (1/5, 0.12) {$\scriptstyle a$};
	\node at (.4, 0.12) {$\scriptstyle f(c_1)$};
	\node at (.4, 0.05) {$\scriptstyle =g(\bar{c}_1)$};
	\node[circle,fill, inner sep=1] at (0.4,0.2){};
	\node at (3/5, 0.12) {$\scriptstyle b$};
	\end{tikzpicture}
	\caption{Left: small perturbations of a map $f\in \mathrm{PA}_{\lambda}(I)$;
	Right: $g^{\ell-1} =f^{\ell-1}$.}\label{fig:Pert}
\end{figure}
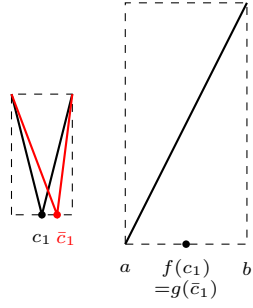
The critical set $\crit(f)$ is finite since $f$ is piecewise affine.
If this perturbation is small enough to be disjoint from the set $\mathcal{U} := \cup_{i=1}^{\ell -1} f^i(U)$
then the
resulting map $g| \mathcal{U} = f| \mathcal{U}$, and so $g^{\ell -1}| U = f^{\ell-1}| U$.  Furthermore, if the perturbation is small enough
so that $g(\bar{c}_1) \in U$ then
 $$g^{\ell-1}\circ g (\bar{c}_1) = f^{\ell - 1} \circ  g (\bar{c}_1) 
= f^{\ell -1} \circ  f(c_1) = c_1.$$
 Moreover, we can choose the perturbation so small that these two points are arbitrarily close,
i.e.,
 $g^{\ell} (\bar{c}_1) \in (  c_1 - \varepsilon, c_1)$, for any fixed $\varepsilon > 0$.

If $\varepsilon$ is small enough then there are no critical points of $f$ in the interval
 $(c_1 - \varepsilon,c_1)$. Furthermore, since $\bar{c}_1 > c_1$, then if the perturbation
 and $\varepsilon$ are small enough it holds that $\bar{c}_1$ is also not in this interval.
This procedure possibly creates new critical connections but of length at least $\ell+1$; but inductive application of both cases will eventually get rid of the critical connections of length at most $k$.

 If we choose the interval of perturbation small enough then no new critical connections of length at most $k$ can be created, thus the proof of the claim is finished. 
\end{proof}

Now we have prepared all the tools to give the proof of Theorem~\ref{t8}.
In what follows, by $\underline{\mathrm{dim}}_{Box}$, $\overline{\mathrm{dim}}_{Box}$ and $\mathrm{dim}_H$ we denote the lower box dimension, the upper box dimension and the Hausdorff dimension of the underlying sets respectively.

\begin{proof}[Proof of Theorem~\ref{t8}]
First note that \ref{pp5}) follows from  \ref{pp3}) since 
 $$\mathrm{dim}_H(\per(f)) \le \sup_{k \ge 1} \mathrm{dim}_H(\fix(f,k)) \le \sup_{k \ge 1} \underline{\mathrm{dim}}_{Box}(\fix(f,k)) = 0.$$

For the proofs of \ref{pp1}), \ref{pp2}), \ref{pp3}) and \ref{pp4}) we fix $k \in \N$.

 Thus we can choose a  countable set  \{$g_i\}_{i\geq 1} \subset \mathrm{PA}_{\lambda}(I)$, such that no $g_i$ has slope $\pm 1 $ on any interval, which is dense in $C_{\lambda}(I)$  with
each $g_i$ satisfying \eqref{e1}.  The advantage of such $g_i$ is that for each point in $\fix(g_i,k)$, there is at least one  corresponding periodic
point in $\fix(g,k)$ if  the perturbed map $g$  if is sufficiently close to $g_i$.

Consider the shortest length
$$\gamma_i:=\min\{|c-c'|: c,c'\in \crit(g_i) \cup \{0,1\} \text{ and } c\neq c'\}$$
of the intervals of monotonicity of $g_i$, note that $\gamma_i>0$ since $g_i\in \mathrm{PA}_{\lambda}(I)$. 

Since $g_i \in  \mathrm{PA}_{\lambda}(I)$ do not have slope $\pm 1$ the set $\fix(g_i,k)$ is finite, suppose it consists of $\ell_i$ disjoint orbits and  the set $\per(g_i,k)$ consists of $\bar \ell_i \le \ell_i$ distinct orbits.
In particular 
\begin{equation}\label{e-count}
\ell_i \le \# \fix(g_i,k) \le k \ell_i \text{  and  } \per(g_i,k) = k \bar \ell_i.
\end{equation}
Choosing one point  from each of the orbits  in $\fix(g_i,k)$ defines the set
 $\{x_{l,i}: 1 \le l \le \ell_i \} \subset \fix(g_i,k)$.

By the definition of $g_i$, the minimal distance 
$$\eta_i:=\min\{|g_i^{m}(x_{l,i})-c|:  0 \le m \le k(x_{l,i})-1, \ \ 1 \le l \le \ell_i , \ c \in \crit(g_i) \cup \{0,1\} \}$$  of the periodic orbits to the 
set $\crit(g_i) \cup \{0,1\}$ is strictly positive.

 If $k=\ell_i =1$ let  $\beta_i := 1$, otherwise we consider the minimal distance
 $$\beta_i  := \frac{1}{2}  \min\{|x - x'|: x \ne x' \in 
 \fix(g_i,k)\}.$$

Let $\tau_i$ be a positive real number such that the slope of every $|(g_i^k)'(x)| < \tau_i$  for every point $x$
where $g_i^k$ is differentiable.

The construction in the proof depends on integers $n_i \ge 1$ which will be defined in the proof, for most of the estimates it suffices to have
$n_i =1$, but for the upper box dimension estimates we will need $n_i$ growing sufficiently quickly.
We define a new map $h_ i\in  \mathrm{PA}_{\lambda}(I)$ by applying a regular $2n_i + 1$-fold window perturbation of $g_i$ of diameter $ a_i \le \frac{1}{2\tau_i} \min(\frac{1}{i k \ell_i},\eta_i,\gamma_i,\beta_i, (k\ell_i)^{-i})$
around each of the points $x_{l,i}$  keeping the map $g_i$ unchanged  elsewhere, in particular it
is unchanged around the other points in $\fix(g_i,k)$.
The perturbations are disjoint  from one another (perturbation around $x_{l,i}$ and $x_{l',i}$) by the definition of $a_i$. 
 The bound on $a_i$  guarantees that these maps satisfy the following properties:
\begin{enumerate}[i)]
\item The collection $\{h_i\}_{i\geq 1}$ is dense in $C_{\lambda}(I)$ (since the total perturbation size is bounded by  $k \ell_i a_i \to 0$);
\item \label{p2} Suppose $x_{l,i} \in \fix(g_i,k)$. 
\begin{enumerate}[a)]
\item
 The map $h_i^{k(x_{l,i})}$  has exactly  $2n_i + 1$ fixed points
in the interval  $I_{l,i} := [x_{l,i} -a_i, x_{l,i} + a_i]$,
\item \label{p2''} The map $h_i^k$ has
 $(2n_i+1)^{k/k(x_{l,i})}$ fixed points in this interval.
 \item \label{p2'} The full branches of $h_i^{k/k(x_{l,i})}$ have length  $a_i/(2n_i+1)^{k/k(x_{l,i})}$, thus
each subinterval of $I_{l,i}$ of  length $2a_i/(2n_i+1)^{k/k(x_{l,i})}$ contains at least one full branch and at most parts of three branches,  and thus at least one fixed point and at most 3 fixed point of $h_i^k$.
 \end{enumerate}
 \item \label{p3} The total number $N_{l,i}$ of fixed points of $h_i^k$ arising from the orbit of $x_{l,i}$  satisfies
$$ N_{l,i}   =  (2n_i + 1)^{k/k(x_{l,i})} k(x_{l,i}).$$
Summing over the points $x_{l,i}$  and using $1 \le k(x_{l,i}) \le k$ yields
$$\max( (2n_i + 1) \ell_i ,  (2n_i + 1)^k)
 \le \#  \fix(h_i,k) = \sum_{l=1}^{\ell_i} N_{l,i}   \le (2n_i + 1)^{k} k \ell_i.$$
\item If $x_{l,i} \in \per(g_i,k)$ (i.e., $k(x_{l,i}) = k$) then the $N_{l,i} = (2n_i + 1)k$ points are
not only in $\fix(h_i,k)$ but also in  $\per(h_i,k)$;
thus  $\# \per(h_i,k) \ge (2n_i + 1)k \bar \ell_i$;
\item\label{p5} Any interval of length $a_i/(2n_i+1)^k$ covers  at most two points of $\fix(h_i,k)$
(since $h_i$ restricted to an interval of length  $a_i/(2n_i+1)$ has at most one critical
point).
\end{enumerate}

Consider $\delta_i > 0$ and
 $$G:= \bigcap_{j \in\N} \bigcup_{i \ge j} B(h_i ,\delta_i).$$
 The set $G$ is a dense $G_\delta$ set.
 
 (\ref{pp1}) We claim that if $\delta_i > 0$ goes to zero sufficiently quickly then $\fix(f,k)$ is a Cantor set for each $f \in G$.
 The set $\fix(h_i,k)$ is finite, choose $\zeta_i$ so small that the balls of radius $\zeta_i$ around distinct points of $\fix(h_i,k)$ are 
disjoint and such that $\zeta_i \to 0$.
We can choose $\delta_i$ so small that if $f \in B(h_i,\delta_i)$ then $\fix(f,k) \subset  B(\fix(h_i,k),\zeta_i)$;  in particular the set
$\fix(f,k)$ can not contain an interval whose length is longer than $2\zeta_i$.
 Fix $f \in G$,  thus $f \in B(h_{i_j},\delta_{i_j})$  for some subsequence $i_j$.
Since $\zeta_{i _j}\to 0$ $\fix(f,k)$ can not contain an interval.
 
By  its definition the set $\fix(f,k)$ is closed. 
Consider the open cover of $\fix(h_i,k)$ by pairwise disjoint intervals of length 
$a_{i_j}$,  by \eqref{p2} these intervals contain $(2n_i+1)^{k/k(x_{l,i})}$ fixed points for some $k(x_{l,i})$.
 Fix these covering intervals and  choose $\delta_i$ 
sufficiently small so that all fixed points of $f^k$ of any $f \in B(h_{i},\delta_{i})$ are contained 
in  the covering intervals and so that there are at least  $ \fix(h_i,k) \ge \#(2n_i + 1) \ell_i$ such points (\ref{p3}). 
Thus
there are no isolated points in $\fix(f,k)$ since for any periodic point of period $k$ we can find another  point from $\fix(f,k)$ arbitrary close.
This completes the proof of \eqref{pp1}.

To prove (\ref{pp2}) we need to make some  adjustments to the above argument.
We again fix $f \in G$,  thus $f \in B(h_{i_j},\delta_{i_j})$  for some subsequence $i_j$.
For the same reasons as before the set $\per(f,k)$ is closed and
can not contain any intervals. 
We consider the same open cover as above and impose the same restrictions on
 $\delta_i$ as above.

We claim that if $x_{l,j} \in \fix(g_i,k)$ then $\# \per(h_i,k) \cap I_{l,i} > 1$.
If $x_{l,i} \in \per(g_i,k)$ has period $k$ then the $(2n_i+1)$ fixed points of $h_i$ in the interval $I_{l,i}$ all have period $k$ and so the claim holds in this case.
Suppose now $x_{l,i} \in \fix(g_i,k)$ has period $k(x_{l,i})$ (a strict divisor of $k$), then the corresponding $(2n_i +1)$ points in $\fix(h_i,k)
\cap I_{l,i}$ have period $k(x_{l,i})$. If we consider the map $h_i^{k(x_{l,i})}$ restricted to $I_{l,i}$ then it is a $(2n_i+1)$-fold tent map, thus it has periodic orbits of all periods.  In particular,  
periodic points of $h_i^{k(x_{l,i})}$ with period $k/k(x_{l,i})$
belongs to the set $\per(h_i,k)$. The number  
 of such periodic points is strictly larger than $1$, and the claim follows.

We additionally require that  $\delta_i$ is
sufficiently small so that  for any $f \in B(h_{i},\delta_{i})$
not only are 
all the  fixed points of $f^k$ contained 
in  the covering intervals but also that $f$ restricted to the covering interval around $x_{l,i_j}$ has
at least 2 periodic points of period $k$.

Thus
there are no isolated points in $\per(f,k)$ since for any periodic point of period $k$ we can find another  point from $\per(f,k)$ arbitrary close.
This completes the proof of \eqref{pp2}.

(\ref{pp3}) Since $\per(f,k) \subset \fix(f,k)$ it suffices to prove the statement for $\fix(f,k)$.
Remember that the number $k \ge 1$ and  the sequence $\ell_i$  are fixed.
We claim that if  $\delta_i > 0$ goes to zero sufficiently quickly   then
the lower  box dimension of the $\fix(f,k)$ is zero for any $f \in G$.

To prove the claim fix $f \in G$, thus $f \in B(h_{i_j},\delta_{i_j})$  for some subsequence $i_j$.
Consider the open cover of $\fix(h_{i_j},k)$ by intervals of length 
$a_{i_j}$ guaranteed by \eqref{p2}. Fix these covering intervals and  choose $\delta_{i_j}$ sufficiently small so that all points of $\fix(f,k)$ of any $f \in B(h_{i_j},\delta_{i_j})$ are contained in  the covering intervals.

Let $N(\eps)$ denote the number of intervals of length $\eps > 0$ needed to cover $\fix(h_{i_j},k)$.  By the choice of $\delta_{i_j}$, these intervals of length $a_{i_j}$
also cover $\fix(f,k)$.
Equation \eqref{e-count} combined with \eqref{p2}  implies  that $\ell_{i_j} \le N(a_{i_j}) \le k \ell_{i_j}$.
Combining this with the fact that $a_{i_j} \le (k \ell_{i_j})^{-i_j}$ yields
$$\frac{\log (N(a_{i_j})) }{\log(1/a_{i_j})}  \le  \frac{ \log(k \ell_{i_j})}{\log(1/a_{i_j})} \le \frac{1}{i_j}$$
and thus the lower box dimension of $\fix(f,k)$
defined as 
$$\liminf_{\eps \to 0} \frac{\log (N(\eps)) }{\log(1/\eps)}$$
is 0.

(\ref{pp4}) We begin by calculating the upper box dimension of $\fix(f,k)$. Here we will need to choose the sequence $n_i$ growing sufficiently quickly.
 Instead of covering $\fix(h_i,k)$ by intervals of length $a_i$  we cover it by intervals of length $b_i := 2 a_i/(2n_i+1)^k$.  
By \eqref{p2'} each such interval covers at most three points of $\fix(h_i,k)$.
Thus we need at least $(\# \fix(h_i,k))/3$ such intervals to cover $\fix(h_i,k)$; so by \eqref{p3} we need
 at least $(2n_i+1)^{k}/3$ such intervals to cover $\fix(h_i,k)$.
Fix such a covering and  choose $\delta_i$ sufficiently small so that all periodic points  of period $k$ of any $f \in B(h_{i_j},\delta_{i_j})$ are contained in  the covering intervals.

Thus 
\begin{equation}\label{ubd}
\frac{\log (N(b_{i_j})) }{\log(1/b_{i_j})} \ge  \frac{ \log((2n_i+1)^k/3)}{\log(1/b_{i_j})} 
=   \frac{ \log((2n_i+1)^k) - \log(3)}{\log((2n_i+1)^k) - \log(2a_i)}.
\end{equation}

The sequence $a_i$ has been fixed above, 
if $n_i$ grows sufficiently quickly the last term in \eqref{ubd} approaches one. 
We can not cover  $\fix(h_i,k) $ by fewer intervals, and thus we can not cover $\fix(f,k)$ by fewer interval, thus
 it follows that the upper box dimension of $\fix(f,k)$
defined as 
$$\limsup_{\eps \to 0} \frac{\log (N(\eps)) }{\log(1/\eps)}$$
is 1.

{(5)} We modify the above proof to calculate the upper box dimension of  $\per(f,k)$ for $k \ge 2$.
Instead of $\fix(h_i,k)$  we consider $\per(h_i,k)$. As before, an interval of length $b_i$ covers at most three points
of this set, thus we need at least $\#(\per(h_i,k))/3$ such intervals to cover it.  Let $x_{l,i}$ be a fixed point and $I_{l,i}$ the
associated interval, then
\begin{eqnarray*}
\#(\per(h_i,k)) &\ge& \#(\fix(h_i,k) \cap I_{l,i}) - \sum_{\ell|k, 1 \le \ell < k}  \#(\fix(h_i,\ell) \cap I_{l,i})\\
& = & (2n_i +1)^k  - \sum_{\ell|k, 1 \le \ell < k}   (2n_i +1)^\ell\\
& \ge & (2n_i +1)^k  - \sum_{\ell=1}^{\lfloor k/2 \rfloor}   (2n_i +1)^\ell\\
& = & (2n_i+1)^k - \frac{(2n_i+1)^{1+ \lfloor k/2 \rfloor} - 1}{2n_i}\\
& = & (2n_i+1)^k \cdot \left [
1 - \frac{1}{2n_i}  \left (  \frac{1}{  (2n_i+1)^{k - 1 - \lfloor k/2 \rfloor }} + \frac{ 1}{(2n_i+1)^k} \right )
 \right ].
 \end{eqnarray*} 
If we additionally suppose that  $n_i \ge 2$ then for any $k \ge 2$ we have
\begin{eqnarray*}
1 - \frac{1}{2n_i}  \left (  \frac{1}{  (2n_i+1)^{k - 1 - \lfloor k/2 \rfloor }} + \frac{ 1}{(2n_i+1)^k} \right )
& \ge  &  1 - \frac{1}{4}  \left (  \frac{1}{  (2n_i+1)^{k - 1 - \lfloor k/2 \rfloor }} + \frac{ 1}{(2n_i+1)^k} \right ) \ge \frac{3}{4}.
 \end{eqnarray*} 
Thus the estimate \eqref{ubd} becomes
\begin{eqnarray*}
\frac{\log (N(b_{i_j})) }{\log(1/b_{i_j})} &\ge&  \frac{ \log(\frac{3}{4} (2n_i+1)^k /3) }{\log(1/b_{i_j})} \\
&=&   \frac{ \log( (2n_i+1)^k) - \log(4) }{\log((2n_i+1)^k) - \log(2a_i)}
\end{eqnarray*}
and the rest of the proof follows in a similar manner.
\end{proof}

\begin{remark}
The proof of Theorem~\ref{t8} can easily be adapted to show that the generic map in $C(I)$ has the same properties, this does not seem to be known in our setting.
Related results have been proven for homeomorphisms on manifolds of dimension at least two in \cite{PAG} (unpublished sketch) and \cite{PV}.
\end{remark}

While positive Lebesgue measure of periodic points can not be realized for ergodic maps, it turns out it can be visible in many leo Lebesgue measure preserving maps. To this end let us first introduce some needed definitions.

Let $M_f(I)$ be the space of invariant
Borel probability measures on $I$ equipped with the {\it Prohorov metric} $D$ defined by
$$
D(\mu, \nu)=\inf\left\{\eps\colon
\begin{array}{l l}
& \mu(A) \le \nu(B(A,\eps))+\eps \text{ and }
 \nu(A) \le \mu(B(A,\eps))+\eps \\
& \text{ for any Borel subset } A \subset I
\end{array}
\right\}
$$
for $\mu, \nu \in M_f(I)$. The following (asymmetric) formula
$$D(\mu, \nu)=\inf\{\eps\colon
\mu(A)\leq \nu(B(A,\eps))+\eps \text{ for all Borel subsets } A\subset I\}$$
is equivalent to original definition, which means we need to check only one of the inequalities.
It is also well known, that the topology induced by $D$ coincides with the {\it weak$^*$-topology} for measures, in particular $(M_f(I), D)$ is a compact metric space (for more details on Prohorov metric and weak*-topology the reader is referred to \cite{Huber}).

\begin{lem}\label{lem:nonatomic}
Assume that $\per(f,k)$ is a Cantor set. Fix  $x\in \per(f,k)$ and $\eps>0$. Let $\mu_x$ be the unique $f$-invariant Borel probability measure supported on the orbit of $x \in \per(f,k)$. Then
there is a non-atomic measure $\nu$ supported on $\per(f,k)$ such that $D(\mu_x, \nu)<\eps$.
\end{lem}
\begin{proof}
There exists a Cantor set $C\subset \per(f,k)$ such that $x\in C$ and sets $f^i(C)$ are pairwise disjoint with $\diam(f^i(C))<\eps$ for $i=0,\ldots,k-1$.
Let $\hat{\nu}$ be any non-atomic probability measure on $C$ and put 
$\nu=\frac{1}{k}\sum_{i=0}^{k-1}\hat{\nu}\circ f^i$. Clearly $\nu$ is $f$-invariant. 
Note that $\nu(B(f^i(x),\eps))\geq \nu(f^i(C))=1/k$, which yields that 
$D(\mu_x, \nu)<\eps$.
\end{proof}

\begin{proof}[Proof of Theorem~\ref{t-PP}]
Using Theorem \ref{t8} and the results of \cite{BT} we can choose
a map $f\in C_{\lambda}(I)$ that is leo, ergodic and $\per(f,k)$ is a Cantor set
for each $k$. By result of Blokh, every mixing interval map has the periodic specification property \cite{Blokh} (see also \cite{BT}, Corollary 10). By a well known result of Sigmund \cite{Sig1,Sig2}, so called $\co$-measures, i.e., ergodic measures supported on periodic orbits, are dense in the space of invariant probability measures for maps with periodic specification property.  In our context it means that Lebesgue measure can be approximated arbitrarily well by a $\co$-measure supported on a periodic orbit.
As a consequence, Lemma~\ref{lem:nonatomic} implies that there exists a sequence $\mu_k$ of non-atomic measures supported on a subset of $\per(f)$
such that $\lim_{k\to \infty} D(\mu_k,\lambda)=0$.

Let us fix any $\eps>0$ and without loss of generality assume that $D(\mu_k,\lambda)<\eps$ for every $k$.

Consider the measure 
$$\nu := \sum_{k=1}^{\infty} \frac{1}{2^k}\cdot \mu_k.$$
By definition $\nu$ is an $f$-invariant Borel probability measure, so $f$ preserves both measures $\lambda$ and $\nu$.
As a combination of non-atomic measures, $\nu$ is non-atomic,
and since $\lim_{k\to \infty} D(\mu_k,\lambda)=0$, $\nu$ has full support, i.e., $\supp~\nu=I$.
 
  We define a map $h\colon~I\to I$ by $h(x):=\nu([0,x])$, since $\nu$ has full support and is non-atomic, the map $h$ is a homeomorphism. 
Note that by the definition of the metric $D$ we have
$$
\nu([0,x])\leq \lambda([0,x+\eps])+\eps=x+2\eps
$$
and
$$
x-\eps=\lambda([0,x-\eps])\leq \nu([0,x])+\eps
$$
hence $|x-h(x)|<2\eps$.
  
For each Borel set $A$ in $I$ we can equivalently write
\begin{equation}\label{e:1}
\lambda(h(A))=\nu(A)\text{ or } \lambda(A)=\nu(h^{-1}(A)).
\end{equation}
We claim that $g :=h\circ f\circ h^{-1}\in C_{\lambda}(I)$. Using (\ref{e:1}) for any Borel set $A$ in $I$ we have
\begin{equation*}
  \lambda(A)=\nu(h^{-1}(A))=\nu(f^{-1}(h^{-1}(A)))=
  \lambda(h(f^{-1}(h^{-1}(A))))=\lambda(g^{-1}(A)).
\end{equation*}
Moreover, the maps $g$ and $f$ are topologically conjugated, so the map $g$ is also leo and $h(\per(f))= \per(g)$. But by (\ref{e:1}) again
$$\lambda(\per(g))=\lambda(h(\per(f)))=\nu(\per(f))= \sum_{k=1}^{\infty} \frac{1}{2^k}\cdot \mu_k(\per(f)) = 1.
$$

In the above construction, we may take $\eps$ arbitrarily small, therefore $g$ can be arbitrarily small perturbation of $f$.

Now assume that we are given a leo mao $f\in C_\lambda(I)$ for which $\lambda(\per(f))=1$.
Consider the measure 
$$\eta := \sum_{k=1}^{\infty} \frac{1}{2^k}\cdot \eta_k,$$
where each $\eta_k$ is obtained by application of Lemma~\ref{lem:nonatomic} to a point $x\in \per(f,k)$.
Then $\eta$ is no-natomic and $\eta(\per(f,k))>0$ for every $k$.

We repeat the above proof (construction of map $g$) using measures $\nu_{\eps_i} := 
 \eps_i \cdot \eta + (1 -\eps_i) \cdot \lambda$
where the sequence $0 < \eps_i < 1$ decreases to $0$.
The resulting maps $\{g_{\eps_i}\}$ satisfy  $\rho(g_{\eps_i}, f) \to 0$, completing the proof.
\end{proof}

\subsection{Periodic points for generic circle maps}\label{subsec:PPcirc}
    
Let $C_{\lambda,d}(\mathbb{S}^1)$ denote the set of degree $d$ maps in $C_{\lambda}(\mathbb{S}^1)$. The proof of Theorem \ref{t8} immediately shows:

\begin{thm}
Theorem \ref{t8}
holds for generic  maps in $C_{\lambda,d}(\mathbb{S}^1)$ for each $d \in \mathbb{Z} \setminus \{1\}$.
\end{thm}

For $C_{\lambda,1}(\mathbb{S}^1)$ the situation is more complicated, consider the open set
$$C_p :=\{f \in C_{\lambda,1}(\mathbb{S}^1): f \text{ has a transverse periodic point of period } p\}.$$
In this setting the proof of  Theorem  \ref{t8} yields the following result, 

\begin{thm}\label{thm:C_p}
For any $f$ in a dense $G_\delta$ subset of $\overline{C}_p$ we have that for each
$k \in \N$
\begin{enumerate}
\item the set $\fix(f,kp)$ is  a Cantor set;
\item there exists a Cantor set  $P_{kp} \subset \per(f,kp)$;
\item the sets $P_{kp} \subset \per(f,kp) \subset \fix(f,kp)$ have Hausdorff dimension and  lower box dimension zero, while the upper box dimension of these sets is  one, and
\item the Hausdorff dimension of $\per(f)$ is zero.
\end{enumerate}
\end{thm}

\begin{remark} As in the interval case, the
 proof of the previous two results  can easily be adapted to show that the generic degree $d$ map in $C(\mathbb{S}^1)$ has the same properties, again this does not seem to be known in our setting.
\end{remark}

To interpret this result we investigate the set  $C_{\infty} := C_{\lambda,1}(\mathbb{S}^1) \setminus \cup_{p \ge 1} \overline{C_p}$.  As we already saw in the proof of Theorem \ref{t8}, a periodic point can
be transformed to a transverse periodic point by an arbitrarily small perturbation of the map, thus
the set $C_{\infty}$  consists of maps without periodic points. Using the same argument we see that $\cup_{p \ge 1} \overline{C_p}$ contains an open dense set. Therefore, $C_{\infty}$ is nowhere dense in $C_{\lambda,1}(\mathbb{S}^1)$.

\begin{proposition}
The set $C_{\infty}$ consists of irrational circle rotations.
\end{proposition}

\begin{proof}
Clearly $C_{\infty}$ contains all irrational circle rotations.

We claim that any $f \in C_{\infty}$ must be invertible. 
For each point $z$ denote by $J_z$ the largest interval containing $z$ such that $f^n(z)\not\in J_z$
for all $n>0$.  
Suppose that $f(x)=f(y)$ for some $x \ne y$.
Then by \cite[Theorem~1]{AK} we obtain that $J_x=J_y$, in particular both are nondegenerate intervals.
By the same result intervals $f^n(J_x)$ are pairwise disjoint for all $n\geq 0$.
The Poincar\'e recurrence theorem states that almost every point is recurrent, which is a contradiction since interior of $J_x$
consists of non-recurrent points. The proof is completed.
\end{proof}

\section{Shadowing is generic for Lebesgue measure preserving interval and circle maps}\label{sec:shadowing}

First we recall the definition of shadowing and its related extensions that we will work with in the rest of the paper. For $\delta> 0$, a sequence $(x_n)_{n\in \N_0}\subset I$ s called a \emph{$\delta$-pseudo orbit} of $f\in C(I)$ if $d(f(x_n), x_{n+1})< \delta$ for every $n\in \N_0$. A \emph{periodic $\delta$-pseudo orbit} is a $\delta$-pseudo orbit for which there exists $N\in\N_0$ such that $x_{n+N}=x_n$, for all $n\in \N_0$. We say that the sequence $(x_n)_{n\in\N_0}$ is an \emph{asymptotic pseudo orbit} if $\lim_{n\to\infty} d(f(x_n),x_{n+1})=0$.
If a sequence $(x_n)_{n\in\N_0}$ is a $\delta$-pseudo orbit and an asymptotic pseudo orbit then we simply say that it is an asymptotic $\delta$-pseudo orbit.

\begin{defn}\label{def:shadowing}
We say that a map $f\in C(I)$ has the:
\begin{itemize}
 \item \emph{shadowing property} if for every $\eps > 0$ there exists $\delta >0$ satisfying the following condition: given a $\delta$-pseudo orbit $\mathbf{y}:=(y_n)_{n\in \N_0}$ we can find a corresponding point $x\in I$ which $\eps$-traces $\mathbf{y}$, i.e.,
$$d(f^n(x), y_n)<  \eps \text{ for every } n\in \N_0.$$
\item
\emph{periodic shadowing property} if for every $\eps>0$ there exists $\delta>0$ satisfying the following condition: given a periodic $\delta$-pseudo orbit $\mathbf{y}:=(y_n)_{n\in\N_0}$ we can find a corresponding periodic point $x \in I$, which $\eps$-traces $\mathbf{y}$.
\item \emph{limit shadowing} if for every sequence $(x_n)_{n\in \N_0}\subset I$ so that $$d(f(x_n),x_{n+1})\to 0 \text{ when } n\to \infty$$
there exists $p\in I$ such that
$$d(f^n(p),x_n)\to 0 \text{ as } n\to \infty.$$
\item \emph{s-limit shadowing} if for every $\eps>0$ there exists $\delta>0$ so that
\begin{enumerate}
\item  for every $\delta$-pseudo orbit $\mathbf{y}:=(y_n)_{n\in \N_0}$ we can find a corresponding point $x\in I$ which $\eps$-traces $\mathbf{y}$,
\item  for every asymptotic $\delta$-pseudo orbit $\mathbf{y}:=(y_n)_{n\in \N_0}$ of $f$, there is $x\in I$ which $\eps$-traces $\mathbf{y}$ and
$$ \lim_{n\to \infty}d(y_n,f^n(x)) = 0.$$
\end{enumerate}
\end{itemize}
\end{defn}

 The notions of shadowing and periodic shadowing are classical but let us comment less classical notions of limit and s-limit shadowing. While limit shadowing seems completely different than shadowing, it was proved in \cite{Limit} that transitive maps with limit shadowing also have shadowing property. 
 In general it can happen that for an asymptotic pseudo orbit which is also a $\delta$-pseudo orbit, the point which $\eps$-traces it and the point which traces it in the limit are different \cite{Barwel}. This shows that possessing a common point for such a tracing is a stronger property than the shadowing and limit shadowing properties together and this property introduced in \cite{Sakai} is called the s-limit shadowing.

\begin{obs}\label{rem:implies}
S-limit shadowing implies both classical and limit shadowing.
\end{obs}

\subsection{Proof of genericity of shadowing} The main step in the proof of genericity of 
 the shadowing property in the context of maps from $C_{\lambda}(I)$
 is the following lemma.

\begin{lem}\label{lem:shadowingdense}
For every $\eps>0$ and every map $f\in C_\lambda(I)$
 there are
$\delta<\frac{\eps}{2}$ and $F\in C_{\lambda}(I)$ such that:
\begin{enumerate}
\item $F$ is piecewise affine and $\rho(f,F)<\frac{\eps}{2}$,
\item if $g\in C_\lambda(I)$ and $\rho(F,g)<\delta$ then every $\delta$-pseudo orbit $\mathbf{x}:=\set{x_i}_{i=0}^\infty$ for $g$ is $\eps$-traced by a point $z\in I$.
Furthermore, if $\mathbf{x}$ is a periodic sequence, then $z$ can be chosen to be a periodic point.
\end{enumerate}
\end{lem}
\begin{proof}
\noindent \textbf{Step 1. Partition.}
First, let 	$0<\gamma<\eps/2$  be such that, if $|a-b|<\gamma$ then $|f(a)-f(b)|<\eps/2$.

 Let us assume that $f$ is piecewise affine with the absolute value of the slope at least $4$ on every piece of monotonicity. Indeed, we can assume that $f$ is piecewise affine due to Proposition 8 from \cite{BT}. Furthermore, we can also assume that the absolute value of the slope of $f$ is at least $4$ on every piece of monotonicity by using regular window perturbations from Definition~\ref{def:perturb} and thus we can approximate arbitrarily well any piecewise affine map from  $C_\lambda(I)$ by a piecewise affine map from  $C_\lambda(I)$ having absolute value of the slope at least  $4$ on every piece of monotonicity. We set $\gamma$ to be smaller than the length of the shortest piece of monotonicity of $f$.
 Since $f$  preserves the Lebesgue measure it must have non-zero slope on every interval of monotonicity.
Thus we can assume we have a partition $0=a_0<a_1<\ldots < a_n<a_{n+1}=1$ such that:
\begin{enumerate}
	\item[(i)] $a_{i+1}-a_i\leq \gamma$ for $i=0,\ldots,n$,
	\item[(ii)] if $f(a_i)\not\in \{0,1\}$ then $f(a_i)\neq a_j$ for every $j$.
	\item[(iii)] if $f(a_i)\not\in \{0,1\}$ then $a_i\not\in f(\crit(f))$.
\end{enumerate}

\noindent \textbf{Step 2. Perturbation.}
By the definition of the partition, there is $\delta>0$ such that for each $j=0,\ldots, n$ we have
$$
\{i : f([a_j,a_{j+1}])\cap (a_i,a_{i+1})\neq \emptyset\}= \{i : B(f([a_j,a_{j+1}]),3\delta)\cap (a_i,a_{i+1})\neq \emptyset\}.
$$
We may also assume that $\delta$ is sufficiently small, so that if $f([a_j,a_{j+1}])\cap [a_i,a_{i+1}]\neq \emptyset$ then
\begin{equation}\label{eq:cover1}
f([a_j,a_{j+1}])\supset [a_i, a_i+2\delta] \quad \text{ or }\quad  f([a_j,a_{j+1}])\supset [a_{i+1}-2\delta,a_{i+1}].
\end{equation}
Now, repeating the construction behind Proposition~8 of \cite{BT} we construct a map $F$ by replacing each $f| [a_i,a_{i+1}]$ by its regular $m$-fold window perturbation (see Definition~\ref{def:perturb} and Figure~\ref{fig:perturbations}), with odd $m$
and large enough to satisfy $1/m<\delta$. This way $F$ is still piecewise affine and its minimal slope is larger than the maximal slope of $f$ and such that
\begin{equation}\label{eq:cover2}
F([a_i,a_i+\delta])=F([a_{i+1}-\delta,a_{i+1}])=F([a_i,a_{i+1}])=f([a_i,a_{i+1}]).
\end{equation}
Since $C_{\lambda}(I)$ is invariant under window perturbations we conclude $F\in C_\lambda(I)$.

\noindent \textbf{Step 3. $\eps$-shadowing.}  For some $x\in I$ in what follows denote $\mathrm{dist}(x,J):=\inf\{d(x,y): y\in J\subset I\}$. Also, for an interval $J\subset I$ let $\diam(J):=\sup\{d(x,y): x,y\in J\}$.
Take any $g\in C_\lambda(I)$ such that $\rho(F,g)<\delta$ and let $\mathbf{x}:=\set{x_i}_{i=0}^\infty$ be a $\delta$-pseudo orbit for $g$. We claim that there is a sequence of intervals $J_i$ such that
\begin{enumerate}
\item $\diam J_i \leq \gamma$ and if $i>0$ then $J_{i}\subset g(J_{i-1})$,\label{con:s1}
\item $\mathrm{dist}(x_{i},J_i)< \gamma$,\label{con:s2}
\item for every $i$ there is $p$ such that $F(J_i)=F([a_p,a_{p+1}])$ and $x_i\in [a_p,a_{p+1}]$.\label{con:s3}
\end{enumerate}
Take $p\geq 0$ such that $[a_p,a_{p+1}]\ni x_0$ and put $J_0=[a_p,a_{p+1}]$. Then conditions \eqref{con:s1}--\eqref{con:s3} are satisfied for $i=0$.

Next assume that for $i=0,\ldots,m$ there are intervals $J_i$ such that conditions \eqref{con:s1}--\eqref{con:s3} are satisfied. We will show how to construct $J_{m+1}$.
Denote $F(J_m)=:[a,b]$. By \eqref{con:s3} and the definition of $F$, namely \eqref{eq:cover1} and \eqref{eq:cover2}, there are nonnegative
integers $\hat{i}, \hat{j}$, $\hat{j}-\hat{i}\geq 2$ such that
$$
[a_{\hat{i}+1}-2\delta, a_{\hat{j}-1}+2\delta]\subset [a,b]\subset [a_{\hat{i}}, a_{\hat{j}}].
$$
Furthermore, if $a_{\hat i}\neq 0$ then $a>a_{\hat{i}}+2\delta$ and if $b_{\hat j}<1$ then $b<a_{\hat{j}}-2\delta$. From this it follows that $B([a,b],2\delta)\subset [a_{\hat{i}}, a_{\hat{j}}]$.
Since $\rho(F,g)<\delta$ it holds
\begin{equation}\label{eq:cover3}
[a_{\hat{i}+1}-\delta, a_{\hat{j}-1}+\delta]\subset g(J_m)\subset [a_{\hat{i}}, a_{\hat{j}}]
\end{equation}
and
$$
g(x_m)\in B(F(x_m),\delta)\subset B([a,b],\delta)
$$
and therefore
$$
x_{m+1}\in  B([a,b],2\delta)\subset [a_{\hat{i}}, a_{\hat{j}}].
$$
Then there is $\hat{i}\leq q < \hat{j}$ such that $x_{m+1}\in [a_q,a_{q+1}]$ and if we put $L:=[a_q,a_q+\delta]$ and $R :=[a_{q+1}-\delta, a_{q+1}]$
then by \eqref{eq:cover3} it follows $L\subset g(J_m)$ or $R\subset g(J_m)$. Now, we put $J_{m+1}=L$ or $J_{m+1}=R$ depending on the situation, obtaining that
$J_{m+1}\subset g(J_m)$. Additionally, $\mathrm{dist}(x_{m+1},J_{m+1})<\gamma$ since both $L$ and $R$ are contained in $[a_q,a_{q+1}]$ and by the definition of $F$ it follows from \eqref{eq:cover2} that
\begin{equation*}
F(J_{m+1})=F(L)=F(R)=F([a_q,a_{q+1}]).
\label{def:FLR}    
\end{equation*}
This finishes the inductive construction.

By \eqref{con:s1} there is a point $z\in I$ such that $z\in \bigcap_{i=0}^\infty g^{-i}(J_i)$.
Then $g^i(z)\in J_i$ for every $i\geq 0$ and so by \eqref{con:s1} and \eqref{con:s2} we obtain that
$$
d(g^i(z),x_i)\leq \diam J_i + \mathrm{dist}(x_i,J_i)< 2\gamma<\eps.
$$
We have just proved that the pseudo orbit $\mathbf{x}$ is $\eps$-traced by the point $z$.

To finish the proof, let us assume that additionally $\mathbf{x}$ is periodic. Then we will modify sequence $J_i$
in the following way. Let $N$ be the period of $\mathbf{x}$. Since there are finitely many choices of interval $J_i$, there are nonnegative integers $k<s$
such that $J_{kN}=J_{sN}$. Define a sequence of interval $(L_i)^{\infty}_{i=0}\subset I$ by the formula $L_i=J_{i+kN (\text{mod }(s-k)N)}$. Condition \eqref{con:s1} implies that there is
$z\in \bigcap_{i=0}^\infty g^{-i}(L_i)$. Since sequence $L_i$ is periodic, in particular since $L_{(s-k)N}$ is covered by  $g^{(s-k)N}(L_0)$,
we may select $z$ being a periodic point. But $x_i=x_{i (\text{mod }N)}=x_{i+kN (\text{mod }(s-k)N)}$, so
\eqref{con:s2}  implies that $z$ is $\eps$-tracing $\mathbf{x}$.
\end{proof}

\begin{proof}[Proof of Theorem~\ref{t-pshadow}]
Fix $\{\eps_n\}_{n\in \N}$, where $\eps_n> 0$ and $\eps_n\to 0$ as $n\to \infty$. Let us also fix a dense collection of maps $\{f_k\}_{k\in \N}\subset C_{\lambda}(I)$. Define the set
$$
A_n:=\{f\in C_{\lambda}(I): \exists \delta>0 \text{ so that every }  \delta\text{-pseudo orbit is }  \eps_n\text{-traced}\}.
$$
Let us fix $k,n\in \mathbb{N}$. By Lemma~\ref{lem:shadowingdense} it holds that for every $f\in C_{\lambda}(I)$ and for all integers $s>1/\eps_n$ there exist $F_{k,s}\in C_{\lambda}(I)$ and $\xi_{k,s}>0$ so that $\rho(F_{k,s},f_{k})<1/s$ and $B(F_{k,s},\xi_{k,s})\subset A_n$. Define
$$
Q_n:=\bigcup_{s>\frac{1}{\eps_n}}\bigcup_{k=1}^{\infty} B(F_{k,s},\xi_{k,s})\subset A_n.
$$
Observe that since $f_k$ is in the closure of $Q_n$ for all $k\in\N$ it follows that $Q_n$ is dense in $C_{\lambda}(I)$. Also $B(F_{k,s},\xi_{k,s})$ is an open set and thus $Q_n$ is open in $C_{\lambda}(I)$ as well.
Now, taking the intersection of the collection $\{Q_n\}_{n\in\N}$ we thus get a dense $G_\delta$ set $Q\subset C_{\lambda}(I)$. Clearly, if $f\in Q$ then for every $\eps>0$ there is $\delta>0$ so that every $\delta$-pseudo orbit is $\eps$-traced by some trajectory of $f$ and if $\delta$-pseudo orbit is periodic then such trajectory of $f$ can be required to be periodic as well.
\end{proof}

\subsection{S-limit shadowing for Lebesgue measure preserving interval and circle maps}\label{subsec:s-limit}

In this subsection we address the level of occurrence of the strongest of the above presented notions related with shadowing.

Let us put $\mathrm{LS}_{\lambda}(I):=\{f\in C_{\lambda}(I)\colon~f\text{ has the s-limit shadowing property}\}$.

\begin{proposition}\label{p:1}
	The set $\mathrm{LS}_{\lambda}(I)$ is dense in $C_{\lambda}(I)$.
\end{proposition}

\begin{proof}Choose $\eps>0$. Let $g_0\in \mathrm{PA}_{\lambda(\mathrm{leo})}(I)$. We will show how to perturb $g_0$ to obtain a map $g\in C_{\lambda}(I)$ close to $g_0$ - it will be specified later - which has the limit shadowing property. We will proceed analogously as in the proof of Lemma~\ref{lem:shadowingdense}. In that proof for a given $\eps>0$ a perturbation of $f$ defining $F$ assumes a special finite partition $\P$ and related positive parameters $\gamma,\delta,m$. We will call the whole procedure (such a map $F$)  $(\eps,\P,\gamma,\delta,m)$-perturbation of $f$.
	
	Fix a decreasing sequence $(\eps_n)_{n\ge 1}$ of positive numbers such that
	
	\begin{equation}\label{e:18}\eps_1<\eps~\text{ and }~\eps_n\to 0,~n\to\infty.
	\end{equation}
	
	\noindent {\bf Step~1.} We put $f=g_0$ and consider
	$$F=g_1\text{ as } (\eps_1,\P_1,\gamma_1,\delta_1,m_1)-\text{perturbation of}~f.
	$$
	We assume that $g_0| P^1_i$ is monotone for each $P^1_i\in\P_1$. From Lemma~\ref{lem:shadowingdense} (1) it follows that $\rho(g_0,g_1)<\eps_1/2$ and Lemma~\ref{lem:shadowingdense} (2) implies that for each $g\in B(g_1,\delta_1)$ (hence also for $g_1$ itself) every $\delta_1$-pseudo orbit is $\eps_1$-traced. In addition we can require $\delta_1<\eps/2$.

	\noindent {\bf Step~2.} We put $f=g_1$ and consider
	$$F=g_2\text{ as } (\eps_2,\P_2,\gamma_2,\delta_2,m_2)-\text{perturbation of}~f.
	$$
	We assume that $g_1| P^2_i$ is monotone for each $P^2_i\in\P_2$. Moreover, we choose $\P_2$ to be a refinement of $\P_1$, i.e., each element of $\P_1$ is a union of some elements of $\P_2$. We consider $\gamma_2$ and $\delta_2$ so small that $$B(g_1,\delta_1)\supset B(g_2,\delta_2);$$
	Lemma~\ref{lem:shadowingdense} implies that for each $g\in B(g_2,\delta_2)$ (hence also for $g_2$ itself) every $\delta_2$-pseudo orbit is $\eps_2$-traced.

	\noindent {\bf Step~n.} We put $f=g_{n-1}$ and consider
	$$F=g_n\text{ as }(\eps_n,\P_n,\gamma_n,\delta_n,m_n\})-\text{perturbation of}~f.
	$$
	We assume that $g_{n-1}| P^n_i$ is monotone for each $P^n_i\in\P_n$ and choose $\P_n$ to be a refinement of $\P_{n-1}$. We consider $\gamma_n$ and $\delta_n$ so small that
	\begin{equation}\label{e:11}B(g_{n-1},\delta_{n-1})\supset B(g_n,\delta_n);\end{equation}
	Lemma~\ref{lem:shadowingdense}(2) implies that for each $g\in B(g_n,\delta_n)$ (hence also for $g_n$ itself) every $\delta_n$-pseudo orbit is $\eps_n$-traced.
	
	 The proof of Lemma~\ref{lem:shadowingdense} shows that for a fixed map $g\in B(g_n,\delta_n)$, for every $\delta_n$-pseudo orbit $(x_i)_{i\ge 0}$, if $x_i\in [a^n_{q(i)},a^n_{q(i)+1}]\in\P_n$ for each $i\ge 0$, there exists a sequence of intervals \begin{equation}\label{e:9}J^n_i\in \{[a^n_{q(i)},a^n_{q(i)}+\delta_n],[a^n_{q(i)+1}-\delta_n,a^n_{q(i)+1}]\}\end{equation}
	such that \begin{equation}\label{e:17}g(J^n_{i-1})\supset J^n_i\end{equation}
	and a point $z\in\bigcap_{i=0}^{\infty}g^{-i}(J^n_i)$ satisfies
	\begin{equation}\label{e:8}
	\vert g^i(z)-x_i\vert<\eps_n
	\end{equation}
	for each $i\ge 0$.
	
	By our construction, the convergence of the sequence $(g_n)_{n\geq 0}$ is uniform in $C_{\lambda}(I)$ hence $\lim_{n\to \infty}g_n=G\in C_{\lambda}(I)$. Moreover, since by (\ref{e:11}),
	$$G\in \bigcap_{n}B(g_n,\delta_n),$$ and by the previous the map $G$ has the shadowing property, i.e., for every $\eps>0$ there is $\delta>0$ such that every $\delta$-pseudo orbit is $\eps$-traced.
	
	Let us show that the map $G$ has the s-limit shadowing property. Due to the definition of s-limit shadowing let us assume that a sequence $(x_i)_{i\ge 0}$ is satisfying
    \begin{equation*}\label{e:12}
		\vert G(x_i)-x_{i+1}\vert\to 0,~i\to\infty.
		\end{equation*}
	Obviously there is an increasing sequence $(\ell(n))_{n\ge 1}$ of nonnegative  integers (w.l.o.g. we assume that $\ell(1)=0$, i.e., $(x_i)_{i\ge 0}$ is an asymptotic $\delta_1$-pseudo orbit) such that
	\begin{equation*}\label{e:10}
		\vert G(x_i)-x_{i+1}\vert<\delta_n,~i\ge \ell(n),
		\end{equation*}
	i.e., each sequence $(x_i)_{i\ge \ell(n)}$ is a $\delta_n$-pseudo orbit. Now we repeatedly use the procedure describe after the equation (\ref{e:11}) and containing the equations (\ref{e:9})-(\ref{e:8}). By that procedure, for each $n\in\N$ we can find sequences $(J^n_i)_{i\ge \ell(n)}$ (to simplify our notation on the $n$th level we index $J^n_i$ from $\ell(n)$) such that for each
	\begin{equation}\label{e:13}
		z\in\bigcap_{i=\ell(n)}^{\infty}G^{-i}(J^n_i)\ \ \ G^{\ell(n)}(z)\ \ \ \eps_n-\text{traces}~(x_i)_{i\ge \ell(n)}\text{ for }G.
	\end{equation}
	But by \eqref{eq:cover2} and \eqref{def:FLR} of Step 3 in the proof of Lemma \ref{lem:shadowingdense}, we have $G(J^n_i)=g_n(J^n_i)$ for each $n$ and $i$ and the sequence $(\P_n)_{n\ge 1}$ is nested, so by \eqref{def:FLR} of Step 3 in the proof of Lemma \ref{lem:shadowingdense}, \eqref{e:9} and \eqref{e:17} for each $n$ we get
		\begin{equation}\label{e:14}
		G(J^n_{\ell(n+1)-1})\supset J^{n+1}_{\ell(n+1)}.
		\end{equation}
		If we define a new sequence $(K_i)_{i\ge 0}$ of subintervals of $I$ by
	\begin{equation*}\label{e:15}
	K_i=J^n_i,~\ell(n)\le i\le \ell(n+1)-1,
	\end{equation*}
	then by (\ref{e:14}) the intersection
	\begin{equation*}\label{e:16}
		K=\bigcap_{i=0}^{\infty}G^{-i}(K_i)
		\end{equation*}
	is nonempty. It follows from (\ref{e:13}) and (\ref{e:18}) that for each $z\in K$, $\vert G^i(z)-x_i\vert\to 0$, $i\to\infty$.
	If asymptotic pseudo orbit was $\delta$-pseudo orbit at start, then the choice of intervals $J^1_i$ in the first step ensures $\eps$-tracing.
	
	In order to finish the proof let us recall that we have chosen $\eps_1<\eps$ and $\delta_1<\eps/2$ hence
	$$
	\rho(g_0,G)<\rho(g_0,g_1)+\rho(g_1,G)<\eps/2+\eps/2=\eps.
	$$
	Since the set $\mathrm{PA}_{\lambda(\mathrm{leo})}(I)$ is dense in $C_{\lambda}(I)$, the conclusion of our theorem follows.
\end{proof}

\section*{Acknowledgements}J. Bobok was supported by the European Regional Development Fund, project No.~CZ 02.1.01/0.0/0.0/16\_019/0000778.
J. \v Cin\v c was supported by the FWF Schrödinger Fellowship stand-alone project J 4276-N35.
P. Oprocha was supported by National Science Centre, Poland (NCN), grant no. 2019/35/B/ST1/02239.

\end{document}